\documentclass[11pt]{article}
\usepackage{amsmath}
\usepackage{amssymb}
\usepackage{amsthm}
\numberwithin{equation}{section}
\begin{document}
\title{The boundedness of some singular integral operators on weighted Hardy spaces associated with Schr\"odinger operators}
\author{Hua Wang \footnote{E-mail address: wanghua@pku.edu.cn.}\\
\footnotesize{School of Mathematical Sciences, Peking University, Beijing 100871, China}}
\date{}
\maketitle
\begin{abstract}
Let $L=-\Delta+V$ be a Schr\"odinger operator acting on $L^2(\mathbb R^n)$, $n\ge1$, where $V\not\equiv 0$ is a nonnegative locally integrable function on $\mathbb R^n$. In this paper, we first define molecules for weighted Hardy spaces $H^p_L(w)$($0<p\le1$) associated to $L$ and establish their molecular characterizations. Then by using the atomic decomposition and molecular characterization of $H^p_L(w)$, we will show that the imaginary power $L^{i\gamma}$ is bounded on $H^p_L(w)$ for $n/{(n+1)}<p\le1$, and the fractional integral operator $L^{-\alpha/2}$ is bounded from $H^p_L(w)$ to $H^q_L(w^{q/p})$, where $0<\alpha<\min\{n/2,1\}$, $n/{(n+1)}<p\le n/{(n+\alpha)}$ and $1/q=1/p-\alpha/n$.
\end{abstract}

\section{Introduction}
Let $n\ge1$ and $V$ be a nonnegative locally integrable function defined on $\mathbb R^n$, not identically zero. We define the form $\mathcal Q$ by
\begin{equation*}
\mathcal Q(u,v)=\int_{\mathbb R^n}\nabla u\cdot\nabla v\,dx+\int_{\mathbb R^n}Vuv\,dx
\end{equation*}
with domain $\mathcal D(\mathcal Q)=\mathcal V\times\mathcal V$ where
\begin{equation*}
\mathcal V=\{u\in L^2(\mathbb R^n):\frac{\partial u}{\partial x_k}\in L^2(\mathbb R^n)\mbox{ for }k=1,\ldots,n \mbox{ and }\sqrt {V}u\in L^2(\mathbb R^n)\}.
\end{equation*}
It is well known that this symmetric form is closed. Note also that it was shown by Simon [15] that this form coincides with the minimal closure of the form given by the same expression but defined on $C^\infty_0(\mathbb R^n)$(the space of $C^\infty$ functions with compact supports). In other words, $C^\infty_0(\mathbb R^n)$ is a core of the form $\mathcal Q$.

Let us denote by $L$ the self-adjoint operator associated with $\mathcal Q$. The domain of $L$ is given by
\begin{equation*}
\mathcal D(L)=\{u\in \mathcal D(\mathcal Q):\exists \,v\in L^2\mbox{ such that }\mathcal Q(u,\varphi)=\int_{\mathbb R^n}v\varphi\,dx,\forall\varphi\in\mathcal D(\mathcal Q)\}.
\end{equation*}
Formally, we write $L=-\Delta+V$ as a Schr\"odinger operator with potential $V$. Let $\{e^{-tL}\}_{t>0}$ be the semigroup of linear operators generated by $-L$ and $p_t(x,y)$ be their kernels. Since $V$ is nonnegative, the Feynman-Kac formula implies that
\begin{equation}
0\le p_t(x,y)\le\frac{1}{(4\pi t)^{n/2}}e^{-\frac{|x-y|^2}{4t}}
\end{equation}
for all $t>0$ and $x,y\in\mathbb R^n$.

Since the Schr\"odinger operator $L$ is a self-adjoint positive definite operator acting on $L^2(\mathbb R^n)$, then $L$ admits the following spectral resolution
\begin{equation*}
L=\int_0^\infty \lambda\,dE_L(\lambda),
\end{equation*}
where the $E_L(\lambda)$ are spectral projectors. For any $\gamma\in\mathbb R$, we shall define the imaginary power $L^{i\gamma}$ associated to $L$ by the formula
\begin{equation*}
L^{i\gamma}=\int_0^\infty\lambda^{i\gamma}\,dE_L(\lambda).
\end{equation*}
By the functional calculus for $L$, we can also define the operator $L^{i\gamma}$ as follows
\begin{equation}
L^{i\gamma}(f)(x)=\frac{1}{\Gamma(-i\gamma)}\int_0^\infty t^{-i\gamma-1}e^{-tL}(f)(x)\,dt.
\end{equation}

By spectral theory $\|L^{i\gamma}\|_{L^2\to L^2}=1$ for all $\gamma\in\mathbb R$. Moreover, it was proved by Shen [12] that $L^{i\gamma}$ is a Calder\'on-Zygmund operator provided that $V\in RH_{n/2}$(Reverse H\"older class). We refer the readers to [6,7,14] for related results concerning the imaginary powers of self-adjoint operators.

For any $0<\alpha<n$, the fractional integrals $L^{-\alpha/2}$ associated to $L$ is defined by
\begin{equation}
L^{-\alpha/2}(f)(x)=\frac{1}{\Gamma(\alpha/2)}\int_0^\infty t^{\alpha/2-1}e^{-tL}(f)(x)\,dt.
\end{equation}

Since the kernel $p_t(x,y)$ of $\{e^{-tL}\}_{t>0}$ satisfies the Gaussian upper bound (1.1), then it is easy to check that $\big|L^{-\alpha/2}(f)(x)\big|\le C I_\alpha(|f|)(x)$ for all $x\in\mathbb R^n$, where $I_\alpha$ denotes the classical fractional integral operator(see [17])
\begin{equation*}
I_\alpha(f)(x)=\frac{\Gamma(\frac{n-\alpha}{2})}{2^\alpha\pi^{\frac n2}\Gamma(\frac{\alpha}{2})}\int_{\mathbb R^n}\frac{f(y)}{|x-y|^{n-\alpha}}dy.
\end{equation*}
Hence, by using the $L^p$-$L^q$ boundedness of $I_\alpha$(see [17]), we have
\begin{equation*}
\|L^{-\alpha/2}(f)\|_{L^q}\le C\|I_\alpha(f)\|_{L^q}\le C\|f\|_{L^p},
\end{equation*}
where $1<p<n/\alpha$ and $1/q=1/p-\alpha/n$. For more information about the fractional integrals $L^{-\alpha/2}$ associated to a general class of operators, we refer the readers to [3,9,20].

In [16], Song and Yan introduced the weighted Hardy spaces $H^1_L(w)$ associated to $L$ in terms of the area integral function and established their atomic decomposition theory. They also showed that the Riesz transform $\nabla L^{-1/2}$ is bounded on $L^p(w)$ for $1<p<2$, and bounded from $H^1_L(w)$ to the classical weighted Hardy space $H^1(w)$(see [4,18]).

Recently, in [19], we defined the weighted Hardy spaces $H^p_L(w)$ associated to $L$ for $0<p<1$ and gave their atomic decompositions. We also obtained that $\nabla L^{-1/2}$ is bounded from $H^p_L(w)$ to the classical weighted Hardy space $H^p(w)$(see also [4,18]) for $n/{(n+1)}<p<1$. In this article, we first define weighted molecules for the weighted Hardy spaces $H^p_L(w)$ associated to $L$ and then establish their molecular characterizations. As an application of the molecular characterization combining with the atomic decomposition of $H^p_L(w)$, we shall obtain some estimates of $L^{i\gamma}$ and $L^{-\alpha/2}$ on $H^p_L(w)$ for $n/{(n+1)}<p\le1$. Our main results are stated as follows.

\newtheorem{theorem}{Theorem}[section]
\begin{theorem}
Let $L=-\Delta+V$, $n/{(n+1)}<p\le1$ and $w\in A_1\cap RH_{(2/p)'}$. Then for any $\gamma\in\mathbb R$, the imaginary power $L^{i\gamma}$ is bounded from $H^p_L(w)$ to the weighted Lebesgue space $L^p(w)$.
\end{theorem}
\begin{theorem}
Let $L=-\Delta+V$, $n/{(n+1)}<p\le1$ and $w\in A_1\cap RH_{(2/p)'}$. Then for any $\gamma\in\mathbb R$, the imaginary power $L^{i\gamma}$ is bounded on $H^p_L(w)$.
\end{theorem}
\begin{theorem}
Suppose that $L=-\Delta+V$. Let $0<\alpha<n/2$, $n/{(n+1)}<p\le1$, $1/q=1/p-\alpha/n$ and $w\in A_1\cap RH_{(2/p)'}$. Then the fractional integral operator $L^{-\alpha/2}$ is bounded from $H^p_L(w)$ to $L^q(w^{q/p})$.
\end{theorem}
\begin{theorem}
Suppose that $L=-\Delta+V$. Let $0<\alpha<\min\{n/2,1\}$, $n/{(n+1)}<p\le n/{(n+\alpha)}$, $1/q=1/p-\alpha/n$ and $w\in A_1\cap RH_{(2/p)'}$. Then the fractional integral operator $L^{-\alpha/2}$ is bounded from $H^p_L(w)$ to $H^q_L(w^{q/p})$.
\end{theorem}

\section{Notations and preliminaries}
First, let us recall some standard definitions and notations. The classical $A_p$ weight
theory was first introduced by Muckenhoupt in the study of weighted
$L^p$ boundedness of Hardy-Littlewood maximal functions in [10].
A weight $w$ is a locally integrable function on $\mathbb R^n$ which takes values in $(0,\infty)$ almost everywhere, $B=B(x_0,r_B)$ denotes the ball with the center $x_0$ and radius $r_B$.
We say that $w\in A_1$, if
$$\frac1{|B|}\int_B w(x)\,dx\le C\cdot\underset{x\in B}{\mbox{ess\,inf}}\,w(x)\quad\mbox{for every ball}\;B\subseteq\mathbb R^n.$$
where $C$ is a positive constant which is independent of $B$.

A weight function $w$ is said to belong to the reverse H\"{o}lder class $RH_r$ if there exist two constants $r>1$ and $C>0$ such that the following reverse H\"{o}lder inequality holds
$$\left(\frac{1}{|B|}\int_B w(x)^r\,dx\right)^{1/r}\le C\left(\frac{1}{|B|}\int_B w(x)\,dx\right)\quad\mbox{for every ball}\; B\subseteq \mathbb R^n.$$

Given a ball $B$ and $\lambda>0$, $\lambda B$ denotes the ball with the same center as $B$ whose radius is $\lambda$ times that of $B$. For a given weight function $w$, we denote the Lebesgue measure of $B$ by $|B|$ and the weighted measure of $B$ by $w(B)$, where $w(B)=\int_Bw(x)\,dx.$

We give the following results which will be often used in the sequel.

\newtheorem{lemma}[theorem]{Lemma}
\begin{lemma}[{[5]}]
Let $w\in A_1$. Then, for any ball $B$, there exists an absolute constant $C$ such that $$w(2B)\le C\,w(B).$$ In general, for any $\lambda>1$, we have $$w(\lambda B)\le C\cdot\lambda^{n}w(B),$$where $C$ does not depend on $B$ nor on $\lambda$.
\end{lemma}
\begin{lemma}[{[5]}]
Let $w\in A_1$. Then there exists a constant $C>0$ such that
$$C\cdot\frac{|E|}{|B|}\le\frac{w(E)}{w(B)}$$
for any measurable subset $E$ of a ball $B$.
\end{lemma}

Given a Muckenhoupt's weight function $w$ on $\mathbb R^n$, for $0<p<\infty$, we denote by $L^p(w)$ the space of all functions satisfying
\begin{equation*}
\|f\|_{L^p(w)}=\left(\int_{\mathbb R^n}|f(x)|^pw(x)\,dx\right)^{1/p}<\infty.
\end{equation*}

Throughout this article, we will use $C$ to denote a positive constant, which is independent of the main parameters and not necessarily the same at each occurrence. By $A\sim B$, we mean that there exists a constant $C>1$ such that $\frac1C\le\frac AB\le C$. Moreover, we denote the conjugate exponent of $s>1$ by $s'=s/(s-1).$

\section{Atomic decomposition and molecular characterization of weighted Hardy spaces}
Let $L=-\Delta+V$. For any $t>0$, we define $P_t=e^{-tL}$ and
\begin{equation*}
Q_{t,k}=(-t)^k\left.\frac{d^kP_s}{ds^k}\right|_{s=t}=(tL)^k e^{-tL},\quad k=1,2,\ldots.
\end{equation*}
We denote simply by $Q_t$ when $k=1$. First note that Gaussian upper bounds carry over from heat kernels to their time derivatives.
\begin{lemma}[{[2,11]}]
For every $k=1,2,\ldots,$ there exist two positive constants $C_k$ and $c_k$ such that the kernel $p_{t,k}(x,y)$ of the operator $Q_{t,k}$ satisfies
\begin{equation*}
\big|p_{t,k}(x,y)\big|\le\frac{C_k}{(4\pi t)^{n/2}}e^{-\frac{|x-y|^2}{c_k t}}
\end{equation*}
for all $t>0$ and almost all $x,y\in\mathbb R^n$.
\end{lemma}
Set
\begin{equation*}
H^2(\mathbb R^n)=\overline{\mathcal R(L)}=\overline{\{Lu\in L^2(\mathbb R^n):u\in L^2(\mathbb R^n)\}},
\end{equation*}
where $\overline{\mathcal R(L)}$ stands for the range of $L$. We also set
\begin{equation*}
\Gamma(x)=\{(y,t)\in{\mathbb R}^{n+1}_+:|x-y|<t\}.
\end{equation*}For a given function $f\in L^2(\mathbb R^n)$, we consider the area integral function associated to Schr\"odinger operator $L$(see [1,8])
\begin{equation*}
S_L(f)(x)=\bigg(\iint_{\Gamma(x)}\big|Q_{t^2}(f)(y)\big|^2\frac{dydt}{t^{n+1}}\bigg)^{1/2},\quad x\in\mathbb R^n.
\end{equation*}
Given a weight function $w$ on $\mathbb R^n$, in [16,19], the authors defined the weighted Hardy spaces $H^p_L(w)$ for $0<p\le1$ as the completion of $H^2(\mathbb R^n)$ in the norm given by the $L^p(w)$-norm of area integral function; that is
\begin{equation*}
\|f\|_{H^p_L(w)}=\|S_L(f)\|_{L^p(w)}.
\end{equation*}
In [16], Song and Yan characterized weighted Hardy spaces $H^1_L(w)$ in terms of atoms in the following way and obtained their atomic characterizations.

\newtheorem{defn}[theorem]{Definition}
\begin{defn}[{[16]}]
Let $M\in\mathbb N$. A function $a(x)\in L^2(\mathbb R^n)$ is called a $(1,2,M)$-atom with respect to $w$(or a $w$-$(1,2,M)$-atom) if there exist a ball $B=B(x_0,r_B)$ and a function $b\in\mathcal D(L^M)$ such that

$(a)$ $a=L^M b$;

$(b)$ $supp\,L^k b\subseteq B, \quad k=0,1,\ldots,M$;

$(c)$ $\|(r^2_BL)^kb\|_{L^2(B)}\le r^{2M}_B|B|^{1/2}w(B)^{-1},\quad k=0,1,\ldots,M$.
\end{defn}
\begin{theorem}[{[16]}]
Let $M\in\mathbb N$ and $w\in A_1\cap RH_2$. If $f\in H^1_L(w)$, then there exist a family of $w$-$(1,2,M)$-atoms \{$a_j$\} and a sequence of real numbers \{$\lambda_j$\} with $\sum_{j}|\lambda_j|\le C\|f\|_{H^1_L(w)}$ such that $f$ can be represented in the form $f(x)=\sum_{j}\lambda_ja_j(x)$, and the sum converges both in the sense of $L^2(\mathbb R^n)$-norm and $H^1_L(w)$-norm.
\end{theorem}
\noindent Similarly, in [19], we introduced the notion of weighted atoms for $H^p_L(w)$($0<p<1$) and proved their atomic characterizations.

\begin{defn}[{[19]}]
Let $M\in\mathbb N$ and $0<p<1$. A function $a(x)\in L^2(\mathbb R^n)$ is called a $(p,2,M)$-atom with respect to $w$(or a $w$-$(p,2,M)$-atom) if there exist a ball $B=B(x_0,r_B)$ and a function $b\in\mathcal D(L^M)$ such that

$(a')$ $a=L^M b$;

$(b')$ $supp\,L^k b\subseteq B, \quad k=0,1,\ldots,M$;

$(c')$ $\|(r^2_BL)^kb\|_{L^2(B)}\le r^{2M}_B|B|^{1/2}w(B)^{-1/p},\quad k=0,1,\ldots,M$.
\end{defn}
\begin{theorem}[{[19]}]
Let $M\in\mathbb N$, $0<p<1$ and $w\in A_1\cap RH_{(2/p)'}$. If $f\in H^p_L(w)$, then there exist a family of $w$-$(p,2,M)$-atoms \{$a_j$\} and a sequence of real numbers \{$\lambda_j$\} with $\sum_{j}|\lambda_j|^p\le C\|f\|^p_{H^p_L(w)}$ such that $f$ can be represented in the form $f(x)=\sum_{j}\lambda_ja_j(x)$, and the sum converges both in the sense of $L^2(\mathbb R^n)$-norm and $H^p_L(w)$-norm.
\end{theorem}

For every bounded Borel function $F:[0,\infty)\to\mathbb C$, we define the operator $F(L):L^2(\mathbb R^n)\to L^2(\mathbb R^n)$ by the following formula
\begin{equation*}
F(L)=\int_0^\infty F(\lambda)\,dE_L(\lambda),
\end{equation*}
where $E_L(\lambda)$ is the spectral decomposition of $L$. Therefore, the operator $\cos(t\sqrt L)$ is well-defined on $L^2(\mathbb R^n)$. Moreover, it follows from [13] that there exists a constant $c_0$ such that the Schwartz kernel $K_{\cos(t\sqrt L)}(x,y)$ of $\cos(t\sqrt L)$ has support contained in $\{(x,y)\in \mathbb R^n\times\mathbb R^n:|x-y|\le c_0t\}$. By the functional calculus for $L$ and Fourier inversion formula, whenever $F$ is an even bounded Borel function with $\hat F\in L^1(\mathbb R)$, we can write $F(\sqrt L)$ in terms of $\cos(t\sqrt L)$; precisely
\begin{equation*}
F(\sqrt L)=(2\pi)^{-1}\int_{-\infty}^\infty \hat F(t)\cos(t\sqrt L)\,dt,
\end{equation*}
which gives
\begin{equation*}
K_{F(\sqrt L)}(x,y)=(2\pi)^{-1}\int_{|t|\ge c_0^{-1}|x-y|}\hat F(t)K_{\cos(t\sqrt L)}(x,y)\,dt.
\end{equation*}

\begin{lemma}[{[8]}]
Let $\varphi\in C^\infty_0(\mathbb R)$ be even and $supp\,\varphi\subseteq[-c_0^{-1},c_0^{-1}]$. Let $\Phi$ denote the Fourier transform of $\varphi$. Then for each $j=0,1,\ldots$, and for all $t>0$, the Schwartz kernel $K_{(t^2L)^j\Phi(t\sqrt L)}(x,y)$ of $(t^2L)^j\Phi(t\sqrt L)$ satisfies
\begin{equation*}
supp\,K_{(t^2L)^j\Phi(t\sqrt L)}\subseteq\{(x,y)\in\mathbb R^n\times\mathbb R^n:|x-y|\le t\}.
\end{equation*}
\end{lemma}
For a given $s>0$, we set
\begin{equation*}
\mathcal F(s)=\Big\{\psi:\mathbb C\to\mathbb C \mbox{ measurable}, |\psi(z)|\le C\frac{|z|^s}{1+|z|^{2s}}\Big\}.
\end{equation*}
Then for any nonzero function $\psi\in\mathcal F(s)$, we have the following estimate(see [16])
\begin{equation}
\Big(\int_0^\infty\|\psi(t\sqrt L)f\|^2_{L^2(\mathbb R^n)}\frac{dt}{t}\Big)^{1/2}\le C\|f\|_{L^2(\mathbb R^n)}.
\end{equation}

We are now going to define the weighted molecules corresponding to the weighted atoms mentioned above.
\begin{defn}
Let $\varepsilon>0$, $M\in\mathbb N$ and $0<p\le1$. A function $m(x)\in L^2(\mathbb R^n)$ is called a $w$-$(p,2,M,\varepsilon)$-molecule associated to $L$ if there exist a ball $B=B(x_0,r_B)$ and a function $b\in\mathcal D(L^M)$ such that

$(A)$ $m=L^M b$;

$(B)$ $\|(r^2_BL)^kb\|_{L^2(2B)}\le r^{2M}_B|B|^{1/2}w(B)^{-1/p}, \quad k=0,1,\ldots,M$;

$(C)$ $\|(r^2_BL)^kb\|_{L^2(2^{j+1}B\backslash 2^j B)}\le 2^{-j\varepsilon}r^{2M}_B|2^jB|^{1/2}w(2^jB)^{-1/p}$,

$\qquad k=0,1,\ldots,M,\quad j=1,2,\ldots$.
\end{defn}

Note that for every $w$-$(p,2,M)$-atom $a$, it is a $w$-$(p,2,M,\varepsilon)$-molecule for all $\varepsilon>0$. Then we are able to establish the following molecular characterization for the weighted Hardy spaces $H^p_L(w)$($0<p\le1$) associated to $L$.

\begin{theorem}
Let $\varepsilon>0$, $M\in\mathbb N$, $0<p\le1$ and $w\in A_1\cap RH_{(2/p)'}$ .

\noindent$(i)$ If $f\in H^p_L(w)$, then there exist a family of $w$-$(p,2,M,\varepsilon)$-molecules \{$m_j$\} and a sequence of real numbers \{$\lambda_j$\} with $\sum_j|\lambda_j|^p\le C\|f\|^p_{H^p_L(w)}$ such that $f(x)=\sum_j\lambda_j m_j(x)$, and the sum converges both in the sense of $L^2(\mathbb R^n)$-norm and $H^p_L(w)$-norm.

\noindent$(ii)$ Assume that $M>\frac n2(\frac1p-\frac12)$. Then every $w$-$(p,2,M,\varepsilon)$-molecule $m$ is in $H^p_L(w)$. Moreover, there exists a constant $C>0$ independent of $m$ such that $\|m\|_{H^p_L(w)}\le C$.
\end{theorem}
\begin{proof}
$(i)$ is a straightforward consequence of Theorems 3.3 and 3.5.

$(ii)$ We follow the same constructions as in [8]. Suppose that $m$ is a $w$-$(p,2,M,\varepsilon)$-molecule associated to a ball $B=B(x_0,r_B)$. Let $\varphi\in C^\infty_0(\mathbb R)$ be even with $supp\,\varphi\subseteq[-(2c_0)^{-1},(2c_0)^{-1}]$ and let $\Phi$ denote the Fourier transform of $\varphi$. We set $\Psi(x)=x^2\Phi(x), x\in\mathbb R$.
By the $L^2$-functional calculus of $L$, for every $m\in L^2(\mathbb R^n)$, we can establish the following version of the Calder\'on reproducing formula
\begin{equation}
m(x)=c_\psi\int_0^\infty (t^2L)^M\Psi^2(t\sqrt L)(m)(x)\frac{dt}{t},
\end{equation}
where the above equality holds in the sense of $L^2(\mathbb R^n)$-norm. Set $U_0(B)=2B$, $U_j(B)=2^{j+1}B\backslash2^jB,j=1,2,\ldots$, then we can decompose
\begin{equation*}
\mathbb R^n\times (0,\infty)=\Big(\bigcup_{j=0}^\infty U_j(B)\times(0,2^jr_B]\Big)\bigcup\Big(\bigcup_{j=1}^\infty 2^jB\times(2^{j-1}r_B,2^jr_B]\Big).
\end{equation*}
Hence, by the formula (3.2), we are able to write
\begin{equation*}
\begin{split}
m(x)=&c_\psi\sum_{j=0}^\infty\int_0^{2^jr_B}(t^2L)^M\Psi^2(t\sqrt L)(m\chi_{U_j(B)})(x)\frac{dt}{t}\\
&+c_\psi\sum_{j=1}^\infty\int_{2^{j-1}r_B}^{2^jr_B}(t^2L)^M\Psi^2(t\sqrt L)(m\chi_{2^{j}B})(x)\frac{dt}{t}\\
=&\sum_{j=0}^\infty m^{(1)}_j(x)+\sum_{j=1}^\infty m^{(2)}_j(x).
\end{split}
\end{equation*}

Let us first estimate the terms $\big\{m^{(1)}_j\big\}^\infty_{j=0}$. We will show that each $m^{(1)}_j$ is a multiple of $w$-$(p,2,M)$-atom with a sequence of coefficients in $l^p$. Indeed, for every $j=0,1,2,\ldots$, one can write
\begin{equation*}
m^{(1)}_j(x)=L^M b_j(x),
\end{equation*}
where
\begin{equation*}
b_j(x)=c_\psi\int_0^{2^jr_B}t^{2M}\Psi^2(t\sqrt L)(m\chi_{U_j(B)})(x)\frac{dt}{t}.
\end{equation*}
By Lemma 3.6, we can easily conclude that for every $k=0,1,\ldots,M$, $supp\,(L^kb_j)\subseteq 2^{j+1}B$.
Since
\begin{equation*}
\Big\|\big[(2^{j+1}r_B)^2L\big]^kb_j\Big\|_{L^2(2^{j+1}B)}=\sup_{\|h\|_{L^2(2^{j+1}B)}\le1}\bigg|\int_{2^{j+1}B}\big[(2^{j+1}r_B)^2L\big]^kb_j(x)h(x)\,dx\bigg|.
\end{equation*}
Then it follows from H\"older's inequality and the estimate (3.1) that
\begin{equation*}
\begin{split}
&\bigg|\int_{2^{j+1}B}\big[(2^{j+1}r_B)^2L\big]^kb_j(x)h(x)\,dx\bigg|\\
=&c_\psi(2^{j+1}r_B)^{2k}\bigg|\int_0^{2^jr_B}\int_{2^{j+1}B}t^{2M}L^k\Psi(t\sqrt L)(m\chi_{U_j(B)})(y)\Psi(t\sqrt L)(h)(y)\frac{dydt}{t}\bigg|\\
\le&c_\psi(2^{j+1}r_B)^{2k}(2^jr_B)^{2M-2k}\bigg|\int_0^{2^jr_B}\int_{2^{j+1}B}(t^2L)^k\Psi(t\sqrt L)(m\chi_{U_j(B)})(y)\Psi(t\sqrt L)(h)(y)\frac{dydt}{t}\bigg|\\
\le&c_\psi(2^{j+1}r_B)^{2M}\bigg(\int_0^\infty\big\|(t^2L)^k\Psi(t\sqrt L)(m\chi_{U_j(B)})\big\|^2_{L^2(\mathbb R^n)}\frac{dt}{t}\bigg)^{1/2}\\
&\times\bigg(\int_0^\infty\big\|\Psi(t\sqrt L)(h\chi_{2^{j+1}B})\big\|^2_{L^2(\mathbb R^n)}\frac{dt}{t}\bigg)^{1/2}\\
\le&c_\psi(2^{j+1}r_B)^{2M}\big\|m\chi_{U_j(B)}\big\|_{L^2(\mathbb R^n)}\cdot\big\|h\chi_{2^{j+1}B}\big\|_{L^2(\mathbb R^n)}\\
\le&C\cdot2^{-j\varepsilon}(2^{j+1}r_B)^{2M}|2^{j+1}B|^{1/2}w(2^{j+1}B)^{-1/p}.
\end{split}
\end{equation*}
Hence
\begin{equation*}
\Big\|\big[(2^{j+1}r_B)^2L\big]^kb_j\Big\|_{L^2(2^{j+1}B)}\le C\cdot2^{-j\varepsilon}(2^{j+1}r_B)^{2M}|2^{j+1}B|^{1/2}w(2^{j+1}B)^{-1/p},
\end{equation*}
which implies our desired result. Next we consider the terms $\big\{m^{(2)}_j\big\}^\infty_{j=1}$. For every $j=1,2,\ldots$, we write
\begin{equation*}
\begin{split}
m^{(2)}_j(x)=&c_\psi\int_{2^{j-1}r_B}^{2^jr_B}(t^2L)^M\Psi^2(t\sqrt L)(m)(x)\frac{dt}{t}\\
&-c_\psi\int_{2^{j-1}r_B}^{2^jr_B}(t^2L)^M\Psi^2(t\sqrt L)(m\chi_{(2^{j}B)^c})(x)\frac{dt}{t}\\
=&m^{(21)}_j(x)-m^{(22)}_j(x).
\end{split}
\end{equation*}
To deal with the term $m^{(21)}_j$, we recall that $m=L^Mb$ for some $b\in\mathcal D(L^M)$. Then we have
\begin{equation*}
\begin{split}
m^{(21)}_j(x)&=c_\psi\int_{2^{j-1}r_B}^{2^jr_B}(t^2L)^M\Psi^2(t\sqrt L)(L^Mb)(x)\frac{dt}{t}\\
&=L^Mb^{(21)}_j(x),
\end{split}
\end{equation*}
where
\begin{equation*}
b^{(21)}_j(x)=c_\psi\int_{2^{j-1}r_B}^{2^jr_B}(t^2L)^M\Psi^2(t\sqrt L)(b)(x)\frac{dt}{t}.
\end{equation*}
Since $b(x)=b(x)\chi_{2^{j}B}(x)+\sum_{l=j}^\infty b(x)\chi_{U_l(B)}(x)$. Then we can further write
\begin{equation*}
b^{(21)}_j(x)=b^{(21)}_{1,j}(x)+\sum_{l=j}^\infty b^{(21)}_{lj}(x),
\end{equation*}
where
\begin{equation*}
b^{(21)}_{1,j}(x)=c_\psi\int_{2^{j-1}r_B}^{2^jr_B}(t^2L)^M\Psi^2(t\sqrt L)(b\chi_{2^{j}B})(x)\frac{dt}{t}
\end{equation*}
and
\begin{equation*}
b^{(21)}_{lj}(x)=c_\psi\int_{2^{j-1}r_B}^{2^jr_B}(t^2L)^M\Psi^2(t\sqrt L)(b\chi_{U_l(B)})(x)\frac{dt}{t}.
\end{equation*}
By using Lemma 3.6 again, we have $supp\,(L^kb^{(21)}_{1,j})\subseteq 2^jB$ and $supp\,(L^kb^{(21)}_{lj})\subseteq 2^{l+1}B$ for every $k=0,1,\ldots,M$. Moreover, it follows from Minkowski's integral inequality that
\begin{equation*}
\begin{split}
&\Big\|\big[(2^{j}r_B)^2L\big]^kb^{(21)}_{1,j}\Big\|_{L^2(2^{j}B)}\\
=&c_\psi(2^{j}r_B)^{2k}\Big\|\int_{2^{j-1}r_B}^{2^jr_B}t^{2M}L^{M+k}\Psi^2(t\sqrt L)(b\chi_{2^{j}B})\frac{dt}{t}\Big\|_{L^2(2^{j}B)}\\
\le&c_\psi(2^{j}r_B)^{2k}\int_{2^{j-1}r_B}^{2^jr_B}\big\|(t^2L)^{M+k}\Psi^2(t\sqrt L)(b\chi_{2^{j}B})\big\|_{L^2(2^{j}B)}\frac{dt}{t^{2k+1}}\\
\le&C\big\|b\chi_{2^{j}B}\big\|_{L^2(2^{j}B)}\\
\end{split}
\end{equation*}
\begin{equation*}
\begin{split}
\le&C\sum_{l=0}^{j-1}\big\|b\chi_{U_l(B)}\big\|_{L^2(2^{j}B)}\\
\le&C\sum_{l=0}^{j-1}2^{-l\varepsilon}r_B^{2M}|2^lB|^{1/2}w(2^lB)^{-1/p}.
\end{split}
\end{equation*}
When $0\le l\le j-1$, then $2^lB\subseteq 2^jB$. By using Lemma 2.2, we can get
\begin{equation*}
\frac{w(2^lB)}{w(2^jB)}\ge C\cdot\frac{|2^lB|}{|2^jB|}.
\end{equation*}
Consequently
\begin{equation*}
\begin{split}
&\Big\|\big[(2^{j}r_B)^2L\big]^kb^{(21)}_{1,j}\Big\|_{L^2(2^{j}B)}\\
\le&C\cdot2^{-j[2M-n(1/p-1/2)]}\cdot(2^{j}r_B)^{2M}|2^{j}B|^{1/2}w(2^{j}B)^{-1/p}\sum_{l=0}^{\infty}\frac{1}{2^{l\varepsilon}}\cdot\frac{1}{2^{l(n/p-n/2)}}\\
\le&C\cdot2^{-j[2M-n(1/p-1/2)]}\cdot(2^{j}r_B)^{2M}|2^{j}B|^{1/2}w(2^{j}B)^{-1/p}.
\end{split}
\end{equation*}
On the other hand
\begin{equation*}
\begin{split}
&\Big\|\big[(2^{l+1}r_B)^2L\big]^kb^{(21)}_{lj}\Big\|_{L^2(2^{l+1}B)}\\
= &c_\psi(2^{l+1}r_B)^{2k}\Big\|\int_{2^{j-1}r_B}^{2^jr_B}t^{2M}L^{M+k}\Psi^2(t\sqrt L)(b\chi_{U_l(B)})\frac{dt}{t}\Big\|_{L^2(2^{l+1}B)}\\
\le&c_\psi(2^{l+1}r_B)^{2k}\int_{2^{j-1}r_B}^{2^jr_B}\big\|(t^2L)^{M+k}\Psi^2(t\sqrt L)(b\chi_{U_l(B)})\big\|_{L^2(2^{l+1}B)}\frac{dt}{t^{2k+1}}\\
\le&C(2^{l+1}r_B)^{2k}\big\|b\chi_{U_l(B)}\big\|_{L^2(2^{l+1}B)}\cdot\frac{1}{(2^jr_B)^{2k}}\\
\le&C\cdot2^{-l\varepsilon}(2^{l+1}r_B)^{2M}|2^{l+1}B|^{1/2}w(2^{l+1}B)^{-1/p}.
\end{split}
\end{equation*}
Observe that $2M>n(1/p-1/2)$. Thus, from the above discussions, we have proved that each $m^{(21)}_j$ is a multiple of $w$-$(p,2,M)$-atom with a sequence of coefficients in $l^p$. Finally, we estimate the terms $\big\{m^{(22)}_j\big\}^\infty_{j=1}$. For every $j=1,2,\ldots$, we decompose $m^{(22)}_j$ as follows
\begin{equation*}
\begin{split}
m^{(22)}_j(x)&=c_\psi\sum_{l=j}^\infty\int_{2^{j-1}r_B}^{2^jr_B}(t^2L)^M\Psi^2(t\sqrt L)(m\chi_{U_l(B)})(x)\frac{dt}{t}\\
&=\sum_{l=j}^\infty L^M b^{(22)}_{lj}(x),
\end{split}
\end{equation*}
where
\begin{equation*}
b^{(22)}_{lj}(x)=c_\psi\int_{2^{j-1}r_B}^{2^jr_B}t^{2M}\Psi^2(t\sqrt L)(m\chi_{U_l(B)})(x)\frac{dt}{t}.
\end{equation*}
It follows immediately from Lemma 3.6 that $supp\,(L^k b^{(22)}_{lj})\subseteq 2^{l+1}B$ for every $k=1,2,\ldots,M$ and $l\ge j$. Moreover
\begin{equation*}
\begin{split}
&\Big\|\big[(2^{l+1}r_B)^2L\big]^kb^{(22)}_{lj}\Big\|_{L^2(2^{l+1}B)}\\
= &c_\psi(2^{l+1}r_B)^{2k}\Big\|\int_{2^{j-1}r_B}^{2^jr_B}t^{2M}L^k\Psi^2(t\sqrt L)(m\chi_{U_l(B)})\frac{dt}{t}\Big\|_{L^2(2^{l+1}B)}\\
\le &c_\psi(2^{l+1}r_B)^{2k}(2^lr_B)^{2M-2k}\int_{2^{j-1}r_B}^{2^jr_B}\big\|(t^2L)^k\Psi^2(t\sqrt L)(m\chi_{U_l(B)})\big\|_{L^2(2^{l+1}B)}\frac{dt}{t}\\
\le &c_\psi(2^{l+1}r_B)^{2M}\big\|m\chi_{U_l(B)}\big\|_{L^2(2^{l+1}B)}\\
\le &C\cdot2^{-l\varepsilon}(2^{l+1}r_B)^{2M}|2^{l+1}B|^{1/2}w(2^{l+1}B)^{-1/p}.
\end{split}
\end{equation*}
Therefore, we have showed that each $m^{(22)}_j$ is also a multiple of $w$-$(p,2,M)$-atom with a sequence of coefficients in $\mathit l^p$. This completes the proof of Theorem 3.8.
\end{proof}

\section{Proofs of Theorems 1.1 and 1.2}
\begin{proof}[Proof of Theorem 1.1]
For any $\gamma\in\mathbb R$, since the operator $L^{i\gamma}$ is linear and bounded on $L^2(\mathbb R^n)$, then by Theorems 3.3 and 3.5, it is enough to show that for any $w$-$(p,2,M)$-atom $a$, $M\in\mathbb N$, there exists a constant $C>0$ independent of $a$ such that
$\|L^{i\gamma}(a)\|_{L^p(w)}\le C$. Let $a$ be a $w$-$(p,2,M)$-atom with $supp\,a\subseteq B=B(x_0,r_B)$, $\|a\|_{L^2(B)}\le |B|^{1/2}w(B)^{-1/p}$. We write
\begin{equation*}
\begin{split}
\big\|L^{i\gamma}(a)\big\|^p_{L^p(w)}&=\int_{2B}\big|L^{i\gamma}(a)(x)\big|^pw(x)\,dx+\int_{(2B)^c}\big|L^{i\gamma}(a)(x)\big|^pw(x)\,dx\\
&=I_1+I_2.
\end{split}
\end{equation*}
Set $s=2/p>1$. Note that $w\in RH_{s'}$, then it follows from H\"older's inequality, the $L^2$ boundedness of $L^{i\gamma}$ and Lemma 2.1 that
\begin{align}
I_1&\le\Big(\int_{2B}\big|L^{i\gamma}(a)(x)\big|^2\,dx\Big)^{p/2}\Big(\int_{2B}w(x)^{s'}\,dx\Big)^{1/{s'}}\notag\\
&\le C\|a\|^p_{L^2(B)}\cdot\frac{w(2B)}{|2B|^{1/s}}\notag\\
&\le C.
\end{align}
On the other hand, by using H\"older's inequality and the fact that $w\in RH_{s'}$, we can get
\begin{align}
I_2&=\sum_{k=1}^\infty\int_{2^{k+1}B\backslash 2^k B}\big|L^{i\gamma}(a)(x)\big|^pw(x)\,dx\notag\\
&\le C\sum_{k=1}^\infty\Big(\int_{2^{k+1}B\backslash 2^k B}\big|L^{i\gamma}(a)(x)\big|^2\,dx\Big)^{p/2}\cdot\frac{w(2^{k+1}B)}{|2^{k+1}B|^{1/s}}.
\end{align}
For any $x\in 2^{k+1}B\backslash2^k B$, $k=1,2,\ldots,$ by the expression (1.2), we can write
\begin{equation*}
\begin{split}
\big|L^{i\gamma}(a)(x)\big|&\le C\int_0^\infty e^{-tL}(a)(x)\frac{dt}{t}\\
&\le C\int_0^{r^2_B} e^{-tL}(a)(x)\frac{dt}{t}+C\int_{r^2_B}^\infty e^{-tL}(a)(x)\frac{dt}{t}\\
&=\mbox{\upshape I+II}.
\end{split}
\end{equation*}
For the term I, we observe that when $x\in2^{k+1}B\backslash 2^k B$, $y\in B$, then $|x-y|\ge 2^{k-1}r_B$. Hence, by using H\"older's inequality and the estimate (1.1), we deduce
\begin{align}
\big|e^{-tL}a(x)\big|&\le C\cdot\frac{t^{1/2}}{(2^{k-1}r_B)^{n+1}}\int_B\big|a(y)\big|\,dy\notag\\
&\le C\cdot\frac{t^{1/2}}{(2^{k}r_B)^{n+1}}\|a\|_{L^2(\mathbb R^n)}|B|^{1/2}\notag\\
&\le C\cdot w(B)^{-1/p}\frac{t^{1/2}}{2^{k(n+1)}\cdot r_B}.
\end{align}
So we have
\begin{equation*}
\begin{split}
\mbox{\upshape I}&\le C\cdot\frac{1}{2^{k(n+1)}w(B)^{1/p}}\cdot\frac1{r_B}\int_0^{r^2_B}\frac{dt}{\sqrt t}\\
&\le C\cdot\frac{1}{2^{k(n+1)}w(B)^{1/p}}.
\end{split}
\end{equation*}
We now turn to estimate the other term II. In this case, since there exists a function $b\in\mathcal D(L^M)$ such that $a=L^M b$ and $\|b\|_{L^2(B)}\le r^{2M}_B|B|^{1/2}w(B)^{-1/p}$, then it follows from H\"older's inequality and Lemma 3.1 that
\begin{align}
\big|e^{-tL}a(x)\big|&=\big|(tL)^Me^{-tL}b(x)\big|\cdot\frac{1}{t^M}\notag\\
&\le C\cdot\frac{1}{(2^{k-1}r_B)^{n+1}}\int_B|b(y)|\,dy\cdot\frac{1}{t^{M-1/2}}\notag\\
&\le C\cdot\frac{1}{(2^{k}r_B)^{n+1}}\|b\|_{L^2(\mathbb R^n)}|B|^{1/2}\cdot\frac{1}{t^{M-1/2}}\notag\\
&\le C\cdot\frac{r^{2M-1}_B}{2^{k(n+1)}w(B)^{1/p}}\cdot\frac{1}{t^{M-1/2}}.
\end{align}
Consequently
\begin{equation*}
\begin{split}
\mbox{\upshape II}&\le C\cdot\frac{1}{2^{k(n+1)}w(B)^{1/p}}\cdot r^{2M-1}_B\int_{r^2_B}^\infty\frac{dt}{t^{M+1/2}}\\
&\le C\cdot\frac{1}{2^{k(n+1)}w(B)^{1/p}},
\end{split}
\end{equation*}
where in the last inequality we have used the fact that $M\ge1$. Therefore, by combining the above estimates for I and II, we obtain
\begin{equation}
\big|L^{i\gamma}(a)(x)\big|\le C\cdot\frac{1}{2^{k(n+1)}w(B)^{1/p}},\quad \mbox{when }x\in 2^{k+1}B\backslash2^k B.
\end{equation}
Substituting the above inequality (4.5) into (4,2) and using Lemma 2.1, then we have
\begin{align}
I_2&\le C\sum_{k=1}^\infty\frac{1}{2^{kp(n+1)}w(B)}\cdot w(2^{k+1}B)\notag\\
&\le C\sum_{k=1}^\infty\frac{1}{2^{k[p(n+1)-n]}}\notag\\
&\le C,
\end{align}
where the last series is convergent since $p>n/{(n+1)}$. Summarizing the estimates (4.1) and (4.6) derived above, we complete the proof of Theorem 1.1.
\end{proof}

\begin{proof}[Proof of Theorem 1.2]
Since the operator $L^{i\gamma}$ is linear and bounded on $L^2(\mathbb R^n)$, then by using Theorems 3.3, 3.5 and 3.8, it suffices to verify that for every $w$-$(p,2,2M)$-atom $a$, the function $m=L^{i\gamma}(a)$ is a multiple of $w$-$(p,2,M,\varepsilon)$-molecule for some $\varepsilon>0$, and the multiple constant is independent of $a$. Let $a$ be a $w$-$(p,2,2M)$-atom with $supp\,a\subseteq B=B(x_0,r_B)$. Then by definition, there exists a function $b\in\mathcal D(L^{2M})$ such that $a=L^{2M}(b)$ and $\|(r^2_BL)^kb\|_{L^2(B)}\le r^{4M}_B|B|^{1/2}w(B)^{-1/p}$, $k=0,1,\ldots,2M$. We set $\tilde b=L^{i\gamma}(L^Mb)$, then we have $m=L^M(\tilde b)$. Obviously, $m(x)\in L^2(\mathbb R^n)$. Moreover, for $k=0,1,\ldots,M$, we can deduce
\begin{equation*}
\begin{split}
\big\|(r_B^2L)^k\tilde b\big\|_{L^2(2B)}&=\frac{1}{r^{2M}_B}\big\|L^{i\gamma}[(r^2_BL)^{M+k}b]\big\|_{L^2(2B)}\\
&\le C\cdot\frac{1}{r^{2M}_B}\big\|(r^2_BL)^{M+k}b\big\|_{L^2(B)}\\
&\le C\cdot r^{2M}_B|B|^{1/2}w(B)^{-1/p}.
\end{split}
\end{equation*}

It remains to estimate $\|(r_B^2L)^k\tilde b\|_{L^2(2^{j+1}B\backslash 2^j B)}$ for $k=0,1,\ldots,M$, $j=1,2,\ldots$. We write
\begin{equation*}
\begin{split}
\big|(r^2_BL)^k\tilde b(x)\big|&=\big|L^{i\gamma}[(r^2_BL)^kL^Mb](x)\big|\\
&\le C\int_0^{r^2_B}e^{-tL}[r^{2k}_BL^{M+k}b](x)\frac{dt}{t}+C\int_{r^2_B}^\infty e^{-tL}[r^{2k}_BL^{M+k}b](x)\frac{dt}{t}\\
&=\mbox{\upshape I$^\prime$+II$^\prime$}.
\end{split}
\end{equation*}
As mentioned in the proof of Theorem 1.1, we know that when $x\in2^{j+1}B\backslash 2^j B$, $y\in B$, then $|x-y|\ge2^{j-1}r_B$, $j=1,2,\ldots$. It follows from H\"older's inequality and the estimate (1.1) that
\begin{align}
\mbox{\upshape I$^\prime$}&\le C\int_0^{r^2_B}\frac{t^{1/2}}{(2^{j-1}r_B)^{n+1}}\big\|r^{2k}_BL^{M+k}b\big\|_{L^2(\mathbb R^n)}|B|^{1/2}\frac{dt}{t}\notag\\
&\le C\cdot\frac{1}{(2^jr_B)^{n+1}}\Big(\frac{1}{r^{2M}_B}\cdot r^{4M}_B|B|^{1/2}w(B)^{-1/p}\Big)|B|^{1/2}\int_0^{r^2_B}\frac{dt}{\sqrt t}\notag\\
&\le C\cdot\frac{1}{2^{j(n+1)}}\cdot r^{2M}_Bw(B)^{-1/p}.
\end{align}
Since $B\subseteq 2^jB$, $j=1,2,\ldots$, then by using Lemma 2.2, we can get
\begin{equation}
\frac{w(B)}{w(2^jB)}\ge C\cdot\frac{|B|}{|2^jB|}.
\end{equation}
Hence
\begin{equation*}
\mbox{\upshape I$^\prime$}\le C\cdot\frac{1}{2^{j[(n+1)-n/p]}}\cdot r^{2M}_Bw(2^jB)^{-1/p},\quad \mbox{when }x\in 2^{j+1}B\backslash2^j B.
\end{equation*}
Applying H\"older's inequality and Lemma 3.1, we obtain
\begin{align}
\mbox{\upshape II$^\prime$}&\le C\cdot r^{2k}_B\int_{r^2_B}^\infty(tL)^{M+k}e^{-tL}(b)(x)\frac{dt}{t^{M+k+1}}\notag\\
&\le C\cdot r^{2k}_B\int_{r^2_B}^\infty\frac{1}{(2^{j-1}r_B)^{n+1}}\|b\|_{L^2(\mathbb R^n)}|B|^{1/2}\frac{dt}{t^{M+k+1/2}}\notag\\
&\le C\cdot\frac{1}{(2^jr_B)^{n+1}}\Big(r^{4M+2k}_B|B|w(B)^{-1/p}\Big)\int_{r^2_B}^\infty\frac{dt}{t^{M+k+1/2}}\notag\\
&\le C\cdot\frac{1}{2^{j(n+1)}}\cdot r^{2M}_Bw(B)^{-1/p}.
\end{align}
It follows immediately from the above inequality (4.8) that
\begin{equation*}
\mbox{\upshape II$^\prime$}\le C\cdot\frac{1}{2^{j[(n+1)-n/p]}}\cdot r^{2M}_Bw(2^jB)^{-1/p},\quad \mbox{when }x\in 2^{j+1}B\backslash2^j B.
\end{equation*}
Combining the above estimates for I$^\prime$ and II$^\prime$, we thus obtain
\begin{equation*}
\big\|(r^2_BL)^k\tilde b\big\|_{L^2(2^{j+1}B\backslash 2^j B)}\le C\cdot\frac{1}{2^{j[(n+1)-n/p]}}\cdot r^{2M}_B|2^jB|^{1/2}w(2^jB)^{-1/p}.
\end{equation*}
Observe that $p>n/{(n+1)}$. If we set $\varepsilon=(n+1)-n/p$, then we have $\varepsilon>0$. Therefore, we have proved that the function $m=L^{i\gamma}(a)$ is a multiple of $w$-$(p,2,M,\varepsilon)$-molecule. This completes the proof of Theorem 1.2.
\end{proof}

\section{Proofs of Theorems 1.3 and 1.4}
\begin{proof}[Proof of Theorem 1.3]
As in the proof of Theorem 1.1, it suffices to prove that for every $w$-$(p,2,M)$-atom $a$, $M>{(3n)}/4$, there exists a constant $C>0$ independent of $a$ such that $\|L^{-\alpha/2}(a)\|_{L^q(w^{q/p})}\le C$. We write
\begin{equation*}
\begin{split}
\big\|L^{-\alpha/2}(a)\big\|^q_{L^q(w^{q/p})}&=\int_{2B}\big|L^{-\alpha/2}(a)(x)\big|^qw(x)^{q/p}\,dx+\int_{(2B)^c}\big|L^{-\alpha/2}(a)(x)\big|^qw(x)^{q/p}\,dx\\
&=J_1+J_2.
\end{split}
\end{equation*}
First note that $0<\alpha<n/2$, $1/q=1/p-\alpha/n$, then we are able to choose a number $\mu>q$ such that $1/\mu=1/2-\alpha/n$. Set $s=2/p$, then by a simple calculation, we can easily see that $(q/p)\cdot(\mu/q)'=s'$ and $1-q/\mu=q/{(ps')}$. Applying H\"older's inequality, the $L^2$-$L^\mu$ boundedness of $L^{-\alpha/2}$, Lemma 2.1 and $w\in RH_{s'}$, we can get
\begin{align}
J_1&\le\Big(\int_{2B}\big|L^{-\alpha/2}(a)(x)\big|^{q\cdot\frac \mu q}\,dx\Big)^{q/\mu}\Big(\int_{2B}w(x)^{\frac qp\cdot(\frac \mu q)'}\,dx\Big)^{1-q/\mu}\notag\\
&=\Big(\int_{2B}\big|L^{-\alpha/2}(a)(x)\big|^{\mu}\,dx\Big)^{q/\mu}\Big(\int_{2B}w(x)^{s'}\,dx\Big)^{q/{(ps')}}\notag\\
&\le C\big\|a\big\|_{L^2(\mathbb R^n)}^q\left(\frac{w(2B)}{|2B|^{1/s}}\right)^{q/p}\notag\\
&\le C.
\end{align}

We now turn to deal with $J_2$. Using the condition $w\in RH_{s'}$ and H\"older's inequality, we obtain
\begin{align}
J_2&=\sum_{k=1}^\infty\int_{2^{k+1}B\backslash 2^k B}\big|L^{-\alpha/2}(a)(x)\big|^qw(x)^{q/p}\,dx\notag\\
&\le C\sum_{k=1}^\infty\bigg(\int_{2^{k+1}B\backslash 2^k B}\big|L^{-\alpha/2}(a)(x)\big|^\mu\,dx\bigg)^{q/\mu}\cdot\bigg(\frac{w(2^{k+1}B)}{|2^{k+1}B|^{1/s}}\bigg)^{q/p}.
\end{align}
For any $x\in 2^{k+1}B\backslash2^k B$, $k=1,2,\ldots,$ by the expression (1.3), we can write
\begin{equation*}
\begin{split}
\big|L^{-\alpha/2}(a)(x)\big|&\le C\int_0^\infty e^{-tL}(a)(x)\frac{dt}{t^{1-\alpha/2}}\\
&\le C\int_0^{r^2_B} e^{-tL}(a)(x)\frac{dt}{t^{1-\alpha/2}}+C\int_{r^2_B}^\infty e^{-tL}(a)(x)\frac{dt}{t^{1-\alpha/2}}\\
&=\mbox{\upshape III+IV}.
\end{split}
\end{equation*}
For the term III, it follows immediately from (4.3) that
\begin{equation*}
\begin{split}
\mbox{\upshape III}&\le C\cdot\frac{1}{2^{k(n+1)}w(B)^{1/p}}\cdot \frac1{r_B}\int_0^{r^2_B}\frac{dt}{t^{1/2-\alpha/2}}\\
&\le C\cdot\frac{r^\alpha_B}{2^{k(n+1)}w(B)^{1/p}}.
\end{split}
\end{equation*}
For the other term IV, by the previous estimate (4.4), we thus have
\begin{equation*}
\begin{split}
\mbox{\upshape IV}&\le C\cdot\frac{1}{2^{k(n+1)}w(B)^{1/p}}\cdot r^{2M-1}_B\int_{r^2_B}^\infty\frac{dt}{t^{M+1/2-\alpha/2}}\\
&\le C\cdot\frac{r^\alpha_B}{2^{k(n+1)}w(B)^{1/p}},
\end{split}
\end{equation*}
where the last inequality holds since $M>{(3n)}/4>1/2+\alpha/2$. Combining the above estimates for III and IV, we obtain
\begin{equation}
\big|L^{-\alpha/2}(a)(x)\big|\le C\cdot\frac{r^\alpha_B}{2^{k(n+1)}w(B)^{1/p}},\quad \mbox{when }x\in 2^{k+1}B\backslash2^k B.
\end{equation}
Substituting the above inequality (5.3) into (5,2) and using Lemma 2.1, we can get
\begin{align}
J_2&\le C\sum_{k=1}^\infty\big|2^{k+1}B\big|^{q/\mu}\cdot\bigg(\frac{r^\alpha_B}{2^{k(n+1)}w(B)^{1/p}}\bigg)^q\cdot\bigg(\frac{w(2^{k+1}B)}{|2^{k+1}B|^{1/s}}\bigg)^{q/p}\notag\\
&\le C\sum_{k=1}^\infty\frac{1}{2^{k[q(n+1)-n]}}\notag\\
&\le C,
\end{align}
where in the last inequality we have used the fact that $q>p>n/{(n+1)}$.
Therefore, by combining the above inequality (5.4) with (5.1), we conclude the proof of Theorem 1.3.
\end{proof}

\begin{proof}[Proof of Theorem 1.4]
As in the proof of Theorem 1.2, it is enough to show that for every $w$-$(p,2,2M)$-atom $a$, the function $m=L^{-\alpha/2}(a)$ is a multiple of $w$-$(p,2,M,\varepsilon)$-molecule for some $\varepsilon>0$, and the multiple constant is independent of $a$. Let $a$ be a $w$-$(p,2,2M)$-atom with $supp\,a\subseteq B=B(x_0,r_B)$, and $a=L^{2M}(b)$, where $M>{(3n)/4}>\max\{\frac n2(\frac 1p-\frac12),\frac12+\frac{\alpha}2\}$, $b\in\mathcal D(L^{2M})$. Set $\tilde b=L^{-\alpha/2}(L^Mb)$, then we have $m=L^M(\tilde b)$. It is easy to check that $m(x)\in L^2(\mathbb R^n)$. As before, since $0<\alpha<n/2$, then we may choose a number $\mu>2$ such that $1/\mu=1/2-\alpha/n$. For $k=0,1,\ldots,M$, by using H\"older's inequality, the $L^2$-$L^\mu$ boundedness of $L^{-\alpha/2}$, we obtain
\begin{align}
\big\|(r_B^2L)^k\tilde b\big\|_{L^2(2B)}&\le\frac{1}{r^{2M}_B}\big\|L^{-\alpha/2}[(r^2_BL)^{M+k}b]\big\|_{L^\mu(2B)}|2B|^{1/2-1/\mu}\notag\\
&\le C\cdot\frac{1}{r^{2M}_B}\big\|(r^2_BL)^{M+k}b\big\|_{L^2(B)}|B|^{1/2-1/\mu}\notag\\
&\le C\cdot r^{2M}_B|B|^{1/2+\alpha/n}w(B)^{-1/p}.
\end{align}
Note that $1/q=1/p-\alpha/n$, then a straightforward computation yields that $q/p<(2/p)'$ whenever $0<\alpha<n/2$. By our assumption $w\in RH_{(2/p)'}$, then we have $w\in RH_{q/p}$. Consequently 
\begin{equation*}
w^{q/p}(B)^{p/q}\le C\cdot\frac{w(B)}{|B|^{1-p/q}},
\end{equation*}
which implies
\begin{equation}
w(B)^{-1/p}\le C\cdot|B|^{1/q-1/p}w^{q/p}(B)^{-1/q}.
\end{equation}
Substituting the above inequality (5.6) into (5.5), we can get
\begin{equation}
\big\|(r_B^2L)^k\tilde b\big\|_{L^2(2B)}\le C\cdot r^{2M}_B|B|^{1/2}w^{q/p}(B)^{-1/q}.
\end{equation}

It remains to estimate $\|(r_B^2L)^k\tilde b\|_{L^2(2^{j+1}B\backslash 2^j B)}$ for $k=0,1,\ldots,M$, $j=1,2,\ldots$. We write
\begin{equation*}
\begin{split}
&\big|(r^2_BL)^k\tilde b(x)\big|\\
=&\big|L^{-\alpha/2}[(r^2_BL)^kL^Mb](x)\big|\\
\le&C\int_0^{r^2_B}e^{-tL}[r^{2k}_BL^{M+k}b](x)\frac{dt}{t^{1-\alpha/2}}+C\int_{r^2_B}^\infty e^{-tL}[r^{2k}_BL^{M+k}b](x)\frac{dt}{t^{1-\alpha/2}}\\
=&\mbox{\upshape III$^\prime$+IV$^\prime$}.
\end{split}
\end{equation*}
For the term III$^\prime$, by using the same arguments as in the proof of (4.7), we have
\begin{equation*}
\begin{split}
\mbox{\upshape III$^\prime$}&\le C\cdot\frac{1}{2^{j(n+1)}}\cdot r^{2M}_Bw(B)^{-1/p}\frac{1}{r_B}\int_0^{r^2_B}\frac{dt}{t^{1/2-\alpha/2}}\\
&\le C\cdot\frac{1}{2^{j(n+1)}}\cdot r^{2M+\alpha}_Bw(B)^{-1/p}.
\end{split}
\end{equation*}
For the term IV$^\prime$, we follow the same arguments as that of (4.9) and then obtain
\begin{equation*}
\begin{split}
\mbox{\upshape IV$^\prime$}&\le C\cdot\frac{1}{2^{j(n+1)}}\cdot w(B)^{-1/p}r^{4M+2k-1}_B\int_{r^2_B}^\infty\frac{dt}{t^{M+k+1/2-\alpha/2}}\\
&\le C\cdot\frac{1}{2^{j(n+1)}}\cdot r^{2M+\alpha}_Bw(B)^{-1/p}.
\end{split}
\end{equation*}
Combining the above estimates for III$^\prime$ and IV$^\prime$, we can get
\begin{equation*}
\big|(r^2_BL)^k\tilde b(x)\big|\le C\cdot\frac{1}{2^{j(n+1)}}\cdot r^{2M+\alpha}_Bw(B)^{-1/p},\quad \mbox{when }x\in 2^{j+1}B\backslash2^j B.
\end{equation*}
Since $w\in A_1$, then it follows from the previous inequality (4.8) that
\begin{equation*}
\big|(r^2_BL)^k\tilde b(x)\big|\le C\cdot\frac{1}{2^{j[(n+1)-n/p]}}\cdot r^{2M+\alpha}_Bw(2^jB)^{-1/p},\quad \mbox{when }x\in 2^{j+1}B\backslash2^j B.
\end{equation*}
Similar to the proof of (5.6), we can also show that
\begin{equation*}
w(2^jB)^{-1/p}\le C\cdot|2^jB|^{1/q-1/p}w^{q/p}(2^jB)^{-1/q}.
\end{equation*}
Hence
\begin{equation*}
\big|(r^2_BL)^k\tilde b(x)\big|\le C\cdot\frac{1}{2^{j[(n+1)-n/q]}}\cdot r^{2M}_Bw^{q/p}(2^jB)^{-1/q},\quad \mbox{when }x\in 2^{j+1}B\backslash2^j B.
\end{equation*}
Therefore
\begin{equation}
\big\|(r^2_BL)^k\tilde b\big\|_{L^2(2^{j+1}B\backslash 2^j B)}\le C\cdot\frac{1}{2^{j[(n+1)-n/q]}}\cdot r^{2M}_B|2^jB|^{1/2}w^{q/p}(2^jB)^{-1/q}.
\end{equation}
Observe that $1\ge q>p>n/{(n+1)}$. If we set $\varepsilon=(n+1)-n/q$, then $\varepsilon>0$. Summarizing the estimates (5.7) and (5.8) derived above, we finally conclude the proof of Theorem 1.4.
\end{proof}

\end{document}